\documentclass[12pt]{amsart}
\usepackage{amssymb,amsmath,amsthm,latexsym}
\usepackage{mathrsfs}
\usepackage{a4wide}
\usepackage[usenames,dvipsnames]{color}

\usepackage[utf8]{inputenc}
\usepackage[T1]{fontenc}
\usepackage{dsfont}
\DeclareFontFamily{U}{mathx}{\hyphenchar\font45}
\DeclareFontShape{U}{mathx}{m}{n}{
      <5> <6> <7> <8> <9> <10>
      <10.95> <12> <14.4> <17.28> <20.74> <24.88>
      mathx10
      }{}
\DeclareSymbolFont{mathx}{U}{mathx}{m}{n}
\DeclareFontSubstitution{U}{mathx}{m}{n}
\DeclareMathAccent{\widecheck}{0}{mathx}{"71}
\DeclareMathAccent{\wideparen}{0}{mathx}{"75}

\newtheorem{theorem}{Theorem}[section]

\newtheorem{corollary}[theorem]{Corollary}

\theoremstyle{remark}

\newtheorem*{remark*}{Remark}
\theoremstyle{definition}
\newtheorem{definition}[theorem]{Definition}

\numberwithin{equation}{section}
\makeatother

\newcommand{\vertiii}[1]{{\left\vert\kern-0.25ex\left\vert\kern-0.25ex\left\vert #1 
    \right\vert\kern-0.25ex\right\vert\kern-0.25ex\right\vert}}

\newcounter{smallromans}

\newenvironment{romanenumerate}
{\begin{list}{{\normalfont\textrm{(\roman{smallromans})}}}%
  {\usecounter{smallromans}\setlength{\itemindent}{0cm}%
   \setlength{\leftmargin}{5.5ex}\setlength{\labelwidth}{5.5ex}%
   \setlength{\topsep}{.5ex}\setlength{\partopsep}{.5ex}%
   \setlength{\itemsep}{0.1ex}}}%
{\end{list}}

\newcounter{smallromansdash}

{\end{list}}

\newcounter{bigromans} 
  {\end{list}}

\begin{document}
\title{Vector-valued invariant means revisited once again}
\author[T.~Kania]{Tomasz Kania}

\address{Mathematics Institute,
University of Warwick,
Gibbet Hill Rd, 
Coventry, CV4 7AL, 
England}
\email{tomasz.marcin.kania@gmail.com, t.kania@warwick.ac.uk}

\subjclass[2010]{43A07 (primary), and 46B50 (secondary)} 
\keywords{amenable semigroup, invariant mean, vector-valued mean, principle of local reflexivity, Banach space complemented in bidual}

\thanks{The author acknowledges with thanks funding received from the European Research Council / ERC Grant Agreement No.~291497.}
\begin{abstract}Banach spaces that are complemented in the second dual are characterised precisely as those spaces $X$ which enjoy the property that for every amenable semigroup $S$ there exists an $X$-valued analogue of an invariant mean defined on the Banach space of all bounded $X$-valued functions on $S$. This was first observed by Bustos Domecq (\emph{J.~Math.~Anal.~Appl.}, 2002), however the original proof was slightly flawed as remarked by Lipecki. The primary aim of this note is to present a corrected version of the proof. We also demonstrate that universally separably injective spaces always admit invariant means with respect to countable amenable semigroups, thus such semigroups are not rich enough to capture complementation in the second dual as spaces falling into this class need not be complemented in the second dual.  \end{abstract}
\maketitle

\section{Introduction and the main result}

One of the most beautiful applications of the Markov--Kakutani fixed-point theorem is the existence of an invariant mean on every abelian semigroup $S$. To be more precise, denote by $\ell_\infty(S)$ the Banach space of all bounded, scalar-valued functions on $S$ furnished with the supremum norm. A \emph{left-invariant mean} (respectively, a \emph{right-invariant mean}) on $S$ is a norm-one linear functional $m$ on $\ell_\infty(S)$ such that $\langle m, \mathds{1}_S\rangle = 1$ and for each $s\in S$ and $f\in \ell_\infty(S)$ one has $\langle m, {}_sf \rangle = \langle m, f\rangle$ (respectively, $\langle m, f_s \rangle = \langle m, f\rangle$), where ${}_sf(t)=f(st)$ and $f_s(t)=f(ts)$ ($t\in S$). A norm-one functional on $\ell_\infty(S)$ is then called an \emph{invariant mean} if it is both a left- and right-invariant mean and a semigroup admitting an invariant mean is called \emph{amenable}; this notion was first distilled by Day in his seminal paper \cite{day}. We note in passing that the free group on two generators is a paradigm example of a~group without an invariant mean. Using then the just-introduced terminology we may rephrase the statement from the very first sentence: abelian semigroups are amenable. \smallskip

It is perhaps not too widely known that Pe{\l}czy\'{n}ski had employed vector-valued invariant means \emph{en route} to the proof of the theorem saying that if $X$ is a Banach space, $Y\subseteq X$ is a closed subspace, which is complemented in $Y^{**}$ and there exists a Lipschitz map $r\colon X\to Y$ such that $r(y)=y$ for $y\in Y$, then $Y$ is (linearly) complemented in $X$ (\cite[pp.~61--62]{pelczynski}, see also \cite[Theorem 3.3]{benyamini}). Quite clearly, the possibility of averaging vectors in an infinite-dimensional space is a desirable titbit and so has been considered, for instance, in the theory of functional equations (\cite{badora,bgp,gajda,ger}). \smallskip

Let us then introduce properly the notion that we shall be concerned with in this note.
\begin{definition}Let $X$ be a Banach space, $S$ a semigroup and $C\geqslant 1$. We define an $X$\emph{-valued invariant $C$-mean} on $S$ to be a bounded linear operator $M\colon \ell_\infty(S,X)\to X$ of norm at most $C$ such that for every $x\in X$, $s\in S$ and for all $f\in \ell_\infty(S,X)$ we have
\begin{romanenumerate}
\item $M(x\mathds{1}_S) = x$;
\item $M(f) = M({}_sf) =  M(f_s)$,\end{romanenumerate}
where $\ell_\infty(S,X)$ denotes the Banach space of all bounded $X$-valued functions on $S$, $x\mathds{1}_S$ stands for the function constantly equal to $x$ on $S$ and ${}_sf$ and $f_s$ are defined as before. An $X$\emph{-valued invariant mean on }$S$ is then an $X$-valued invariant $C$-mean for some $C\geqslant 1$.  \end{definition}
\noindent\emph{Note}: We use interchangeably two conventions for denoting elements of $\ell_\infty(S,X)$; sometimes we term them by lower-case letters, however later on it will be more convenient to treat them as indexed tuples $(x_s)_{s\in S}$. We trust that this will not lead to confusion.\smallskip

As observed by Pe{\l}czy\'{n}ski himself, if $S$ is an amenable semigroup and $X$ is complemented in $X^{**}$, then there is an $X$-valued invariant mean on $S$ (actually Pe{\l}czy\'{n}ski worked with abelian semigroups but the proof is verbatim the same under the presence of an invariant mean on a general semigroup). Indeed, let $m$ be a (scalar-valued) invariant mean on $S$. Given $f\in \ell_\infty(S, X)$, we define a map $\widetilde{M}\colon \ell_\infty(S, X)\to X^{**}$ by $$\langle \widetilde{M}f, \varphi\rangle = \big\langle m, \big(\langle \varphi, f(t)\rangle\big)_{t\in S} \big\rangle   \quad (\varphi \in X^*). $$
As $m$ is a norm-one functional, $\widetilde{M}$ is a norm-one bounded linear operator. Let $P$ be a~projection from $X^{**}$ onto the canonical copy of $X$. Then $$M = \kappa_X^{-1}P\widetilde{M}$$ is an $X$-valued invariant $\|P\|$-mean on $S$, where $\kappa_X\colon X\to X^{**}$ denotes the canonical embedding into the second dual.\smallskip

Bustos Domecq (\cite[Theorem 2]{felix}) made a statement asserting that the converse to this statement is also true, that is to say, if for every abelian semigroup $S$ there exists an $X$-valued invariant $C$-mean on $S$, then $X$ is complemented in $X^{**}$ by a projection of norm at most $C$. However, as observed by Lipecki in his  Mathematical
Review (MR1943762) of Bustos Domecq's paper, the proof of this theorem contains a gap, which we believed could not be easily fixed without weakening the conclusion of the theorem. Initially, we have verified this result giving a~different, perhaps unnecessarily intricate proof based on ultrapower techniques. Having communicated this to Bustos Domecq, we were offered a simpler fix with a permission to reproduce it here. The primary aim of this note is to offer a remedy to this problem by providing a correct proof. \smallskip

Since the result itself is of interest to people working in stability theory of functional equations, we believe that a revised proof ought to be available for the future reference.

\begin{theorem}\label{main}Let $X$ be a Banach space and $C\geqslant 1$. Then the following assertions are equivalent.
\begin{romanenumerate}
\item\label{complement} $X$ is complemented in $X^{**}$ by a projection of norm at most $C$;
\item\label{allamenable} for every amenable semigroup $S$ there exists an $X$-valued invariant $C$-mean on $S$;
\item\label{cardinality} for every commutative semigroup $S$ of cardinality $|X^{**}|$ there exists an $X$-valued invariant $C$-mean on $S$.
\end{romanenumerate}\end{theorem}
We have already sketched the implication \eqref{complement} $\Rightarrow$ \eqref{allamenable}, which is also the content of \cite[Theorem~1]{felix}. The implication \eqref{allamenable} $\Rightarrow$ \eqref{cardinality} is immediate. Then it will be enough to demonstrate the implication \eqref{cardinality} $\Rightarrow$ \eqref{complement}. We postpone the proof to the final section.  \smallskip

We also obtain a result which allows averaging with respect to amenable semigroups that are in a sense not too large.  In order to state the result, we require a piece of terminology.\smallskip

For a Banach space $X$ we denote by $d(X)$ the \emph{density} of $X$, that is, the minimal cardinality of a dense set in $X$. In particular, separable Banach spaces have by definition density  $\aleph_0$.

\begin{theorem}\label{dy}Let $X$ be a Banach space, $C\geqslant 1$ and let $\lambda\leqslant d(X)$ be a~cardinal number. Suppose that for every subspace $Y$ of $X$ with $d(Y) \leqslant \lambda$ there exists a complemented subspace $Z_Y\subseteq X$ containing $Y$, which is also complemented in $Z^{**}$ by a~projection of norm at most $C$. Suppose moreover that the assignment $Y\mapsto Z_Y$ is inclusion-preserving. Then for every amenable semigroup $S$ with cardinality at most $\lambda$, there exists an $X$-valued invariant $C$-mean on $S$. \end{theorem}

Let us note that there are numerous natural examples of spaces verifying the hypotheses of this theorem. For example, universally separably injective Banach spaces fall into this class because they are precisely those Banach spaces whose separable subspaces are contained in subspaces isomorphic to $\ell_\infty$ (\cite[Theorem 5.2]{spanyards}) (actually it is important that the Banach--Mazur distances of these subspaces to $\ell_\infty$ are uniformly bounded; a close inspection of the proof reveals that this is possible and that the assignment of a complemented superspace may be arranged to be inclusion preserving). A notable example of such a space is the space $\ell_\infty^c(\Gamma)$ of all bounded, countably supported functions on an uncountable set $\Gamma$ endowed with the supremum norm. That $\ell_\infty^c(\Gamma)$ is not complemented in its second dual was probably first observed by Pe{\l}czy\'{n}ski and Sudakov (\cite{pelsud}). However, it is readily seen that every separable subspace $Y$ of $\ell_\infty^c(\Gamma)$ is contained in $$Z_Y=\{f\in \ell_\infty^c(\Gamma)\colon \text{if }f(\gamma)\neq 0\; (\gamma\in \Gamma),\text{ then } g(\gamma)\neq 0\text{ for some } g\in Y\},$$ a subspace isometrically isomorphic to $\ell_\infty$, which, by the fact that $\ell_\infty \cong \ell_1^*$ isometrically, is then complemented in its second dual by a projection of norm 1. Clearly, if $Y_1\subseteq Y_2$ are separable subspaces of $\ell_\infty^c(\Gamma)$, then $Z_{Y_1}\subseteq Z_{Y_2}$. In particular, Theorem~\ref{dy} applies to spaces that are not necessarily complemented in the second dual themselves. We have thus obtained the following corollary.
\begin{corollary}Let $X$ be a universally separably injective Banach space. Then for every countable, amenable semigroup $S$, there exists an $X$-valued invariant mean on $S$.\end{corollary}

We do not know whether countable, amenable semigroups are sufficient for capturing complementation of separable Banach spaces in the second dual. It would also be interesting to know whether one may replace amenable semigroups by amenable groups in the statement of clause (ii) of Theorem~\ref{main}.\smallskip

\noindent \emph{Acknowledgement}. We are indebted to Rados{\l}aw {\L}ukasik (Katowice) for having brought to our attention the problem of characterising Banach spaces which may be targets of invariant means on amenable semigroups  and for pointing out Lipecki's Mathematical Review of Bustos Domecq's paper (\cite{felix}), where a gap in the proof of \cite[Theorem~2]{felix} was detected. We are also grateful to Wojciech Bielas (Prague) for spotting certain slips in the previous version of this note. Finally, we wish to thank Bustos Domecq for insightful e-mail exchanges and a permission to include his correction to the proof of \cite[Theorem~2]{felix}.

\section{Preliminaries and general remarks}
For a Banach space $X$ we denote by $\kappa_X\colon X\to X^{**}$ the canonical embedding in the second dual that is given by $\langle \kappa_Xx, f\rangle = \langle f, x\rangle$ ($x\in X, f\in X^*$). We say that $X$ \emph{is complemented in} $X^{**}$ if there exists a~bounded linear projection from $X^{**}$ onto $\kappa_X(X)$. Lindenstrauss had observed (\cite{lind}) that $X$ is complemented in $X^{**}$ if and only if $X$ is isomorphic to a complemented subspace of a~dual Banach space. A notable example of a Banach space which is not complemented in its second dual is the space $c_0$ and by Sobczyk's theorem, neither is any other separable Banach space which contains an isomorphic copy of it. On the other hand, if $X$ is a dual space itself, then it is complemented in $X^{**}$ via the so-called \emph{Dixmier projection} being the adjoint of the embedding $\kappa_Y\colon Y\to X^*$. The Dixmier projection has then norm one but this need not be the case for a general non-reflexive space complemented in its second dual. Indeed, van Dulst and Singer (\cite[Theorem~2.1]{vDS}) proved that every non-reflexive space can be renormed in such a way that there is no norm-one projection from the second dual of the renormed space onto the image of the canonical embedding. Finet and Schachermayer improved their result in the class of non-reflexive spaces with separable dual by showing that in such case there exist a renorming and a constant $\delta>1$ such that every projection from the second dual (if there is one) has norm at least $\delta$ (\cite[Corollary~10]{FS}). In particular, one cannot hope that Theorem~\ref{main} may improved further so that the constant $C$ appearing in the statement is equal to 1 under the mere hypothesis of $X$ being complemented in $X^{**}$.\smallskip

Let us state the version of the principle of local reflexivity due to Lindenstrauss and Rosenthal (\cite{lindros}) that we shall require.

\begin{theorem}\label{LR}Let $X$ be a Banach space. Then for every finite-dimensional subspace $F\subset X^{**}$ and each $\varepsilon \in (0,1]$ there exists a linear map $P_F^\varepsilon \colon F\to \kappa_X(X)$ such that
\begin{romanenumerate}
\item $(1-\varepsilon)\|x\|\leqslant \|P_F^\varepsilon x\| \leqslant (1+\varepsilon)\|x\| \quad (x\in F)$;
\item $P_F^\varepsilon x = x$ for $x\in F\cap \kappa_X(X)$.
\end{romanenumerate}
\end{theorem}

\section{Proofs of Theorems~\ref{main} and \ref{dy}}
We are now ready to prove the main result of this note.
\begin{proof}[Proof of Theorem~\ref{main}]Let $X$ be a Banach space that satisfies the hypotheses of Theorem~\ref{main}. Without loss of generality we may suppose that $X$ is infinite-dimensional. Denote by $\mathfrak{F}$ the family of all finite-dimensional subspaces of $X^{**}$; it is then clear that $\mathfrak{F}$ is upwards-directed by inclusion. Since every finite-dimensional subspace of $X^{**}$ is determined by a finite subset of $X^{**}$, we have $|\mathfrak{F}|=|X^{**}|$. Let $$S = \{(F, \varepsilon)\colon F\in \mathfrak{F}, \varepsilon \in (0,1]\}.$$
Then $S$ carries the structure of a commutative (idempotent) semigroup with the neutral element $(\{0\}, 1)$ when furnished with the operation $$(F_1, \varepsilon_1) + (F_2, \varepsilon_2) = \big(F_1+F_2, \min\{\varepsilon_1, \varepsilon_2\}\big)\quad \big((F_i, \varepsilon_i)\in S, i=1, 2\big).$$
Certainly, $|S|=|X^{**}|$. Let $M\colon  \ell_\infty(S,X)\to X$ be an invariant $C$-mean on $S$. For each $(F,\varepsilon)\in S$ let $P_F^\varepsilon$ be a fixed linear operator satisfying clauses (i) and (ii) of Theorem \ref{LR}. Let $\Phi\colon S\times X^{**}\to X$ be the function defined by
$$ \Phi\big((F,\varepsilon), x \big) = \left\{\begin{array}{ll} P_F^\varepsilon x, & x\in F \\ 0, & \text{otherwise}.\end{array} \right.\qquad \big((F,\varepsilon)\in S, x\in X^{**}\big).$$
Then $\Phi(\cdot, x)\in \ell_\infty(S,X)$ for each $x\in X^{**}$ as it follows from clause (ii) of Theorem~\ref{LR} that $\|P_F^\varepsilon x\| \leqslant 2\|x\|$. We then set $$Px =\kappa_X M \big( \Phi(\cdot, x )\big) \qquad (x\in X^{**}),$$
so that $P\colon X^{**}\to \kappa_X(X)$. That $P$ is a linear projection and $\|P\|\leqslant C$ may be demonstrated exactly as in the proof of \cite[Theorem 2]{felix}.\end{proof}

\begin{proof}[Proof of Theorem~\ref{dy}]Let $C\geqslant 1$, $\lambda\leqslant d(X)$ and suppose that $X$ is a Banach space whose every subspace $Y$ with $d(Y)\leqslant \lambda$ is contained in a further subspace $Z_Y\subseteq X$ such that there exist bounded projections $P_Y\colon Z_Y^{**}\to \kappa_{Z_Y}(Z_Y)$ and $Q_Y\colon X\to Z_Y$ with $\|P_Y\|\leqslant C$. Let us consider the family $$\mathscr{Y}=\{Y\subseteq X\colon Y \text{ is a subspace with }d(Y)\leqslant \lambda\}.$$
For each $Y\in \mathscr{Y}$ we may choose, by the hypothesis, a bounded projection from $X$ onto a~subspace $Z_Y$ which contains $Y$ and is complemented in $Z_Y^{**}$ be a projection $P_Y$ of norm at most $C$ in such a way that the map $Y\mapsto Z_Y$ is inclusion-preserving. \smallskip

Let $S$ be an amenable semigroup with cardinality at most $\lambda$ and let $m$ be a scalar-valued invariant mean on $S$. Arguing as in \cite[Theorem~1]{felix}, we infer that for each $Y\in \mathscr{Y}$ the formula $$M_Y=\kappa_{Z_Y}^{-1}P_Y\widetilde{M}_Y$$ defines a $Z_Y$-valued invariant $C$-mean on $S$, where $\widetilde{M}_Y\colon \ell_\infty(S, Z_Y)\to Z_Y^{**}$ is defined by
$$\langle \widetilde{M}_Y(x_s)_{s\in S}, \varphi\rangle = \big\langle m, \big(\langle \varphi, x_t\rangle\big)_{t\in S} \big\rangle   \quad \big(\varphi \in {(Q_Y^*)(X^*)}\big). $$

Let $Y_1, Y_2\in \mathscr{Y}$ be subspaces with $Y_1\subseteq Y_2$. Then $Z_{Y_1}\subseteq Z_{Y_2}$. We \emph{claim} that the means produced using the above procedure are compatible in the sense that for $(x_s)_{s\in S}\in \ell_\infty(S,Y_1)$ we have \begin{equation}\label{piesel}M_{Y_1}(x_s)_{s\in S} = M_{Y_2}(x_s)_{s\in S}.\end{equation}
Indeed, pick $(x_s)_{s\in S}\in \ell_\infty(S,Y_1)$. We have $Z_{Y_1}\subseteq Z_{Y_2}$ and $Z_{Y_1}$ is complemented in $Z_{Y_2}$ by $Q_{Y_1}|_{Z_{Y_2}}$. Consequently, $Z_{Y_1}^*$ may be canonically identified with the image of $(Q_{Y_1}|_{Z_{Y_2}})^*$, that is contained in the image of $Q_{Y_2}^*$. Then
\begin{equation}\label{restriction}M_{Y_2}(x_s)_{s\in S} = M_{Y_2}(Q_{Y_1}x_s)_{s\in S} =  \kappa_{Z_{Y_2}}^{-1}P_{Y_2}\widetilde{M}_{Y_2}(Q_{Y_1}x_s)_{s\in S}.\end{equation}
Let $\varphi\in (Q_{Y_2}^*)(X^*)$. Then $\varphi = (Q_{Y_1}|_{Z_{Y_2}})^*\varphi + (I_{Z_{Y_2}} - Q_{Y_1}|_{Z_{Y_2}})^*\varphi$. Consequently,
$$\begin{array}{lcl}\big\langle m, \big(\langle \varphi, Q_{Y_1}x_t\rangle\big)_{t\in S} \big\rangle &= &\big\langle m, \big(\langle (Q_{Y_1}|_{Z_{Y_2}})^*\varphi + (I_{Z_{Y_2}} - Q_{Y_1}|_{Z_{Y_2}})^*\varphi, Q_{Y_1}x_t\rangle\big)_{t\in S} \big\rangle \\
& = &\big\langle m, \big(\langle (Q_{Y_1}|_{Z_{Y_2}})^*\varphi, Q_{Y_1}x_t\rangle\big)_{t\in S} \big\rangle. \end{array}$$
This means that $\big(\widetilde{M}_{Y_2}(x_s)_{s\in S}\big)|_{Z_{Y_1}^*} = \widetilde{M}_{Y_1}(x_s)_{s\in S}$. We may then continue computations from \eqref{restriction} to arrive at
$$\kappa_{Z_{Y_2}}^{-1}P_{Y_2}\widetilde{M}_{Y_2}(Q_{Y_1}x_s)_{s\in S}=\kappa_{Z_{Y_2}}^{-1}P_{Y_2}\widetilde{M}_{Y_1}(x_s)_{s\in S}=\kappa_{Z_{Y_2}}^{-1}P_{Y_1}\widetilde{M}_{Y_1}(x_s)_{s\in S}=M_{Y_1}(x_s)_{s\in S},$$
which completes the proof of the claim.\smallskip

We are now in a position to define a map $M\colon \ell_\infty(S,X)\to X$ by 
$$M(x_s)_{s\in S} = M_Y (x_s)_{s\in S}\quad \big((x_s)_{s\in S}\in \ell_\infty(S,X)\big), $$
where $Y=\overline{{\rm span}}\{x_s\colon s\in S\}$. By \eqref{piesel}, this map is well-defined. Indeed, it remains to notice that $M$ is additive because then all other properties of an $X$-valued invariant $C$-mean on $S$ will follow in similar fashion. For brevity, let us denote by $[x_s\colon s\in S]$ the closed linear span of $\{x_s\colon s\in S\}\subseteq X$. Take $(x_s)_{s\in S}, (y_s)_{s\in S}\in \ell_\infty(S,X)$. We have
$$\begin{array}{lcl}M(x_s+y_s)_{s\in S} & =& M_{[x_s+y_s\colon s\in S]}(x_s+y_s)_{s\in S}\\
&\stackrel{\eqref{piesel}}{=}& M_{\overline{[x_s\colon s\in S]+[y_s\colon s\in S]}}(x_s+y_s)_{s\in S}\\
&=&M_{\overline{[x_s\colon s\in S]+[y_s\colon s\in S]}}(x_s)_{s\in S}+M_{\overline{[x_s\colon s\in S]+[y_s\colon s\in S]}}(y_s)_{s\in S}\\
&\stackrel{\eqref{piesel}}{=}&M_{[x_s\colon s\in S]}(x_s)_{s\in S}+M_{[y_s\colon s\in S]}(y_s)_{s\in S}\\
& = & M(x_s)_{s\in S} + M(x_s)_{s\in S}\end{array}$$
and so the proof is complete.\end{proof}

\end{document}